\documentclass[12pt,a4paper]{article}
\usepackage[utf8]{inputenc}
\usepackage[T1]{fontenc}
\usepackage{amsmath,amssymb,amsthm,mathtools}
\usepackage{tikz-cd}
\usepackage{hyperref}
\usepackage{geometry}
\usepackage{enumitem}
\usepackage{mathrsfs}
\DeclareMathOperator{\Hom}{Hom}

\theoremstyle{plain}
\newtheorem{theorem}{Theorem}[section]
\newtheorem{lemma}[theorem]{Lemma}
\newtheorem{proposition}[theorem]{Proposition}

\newtheorem{conjecture}[theorem]{Conjecture}
\theoremstyle{definition}

\newtheorem{example}[theorem]{Example}
\newtheorem{remark}[theorem]{Remark}

\DeclareMathOperator{\Qpot}{Q^{\diamond}}
\DeclareMathOperator{\Qact}{Q_{\bullet}}
\DeclareMathOperator{\id}{id}
\DeclareMathOperator{\Map}{Map}
\DeclareMathOperator{\Fun}{Fun}

\DeclareMathOperator{\Ran}{Ran}

\DeclareMathOperator{\Shv}{Sh}
\DeclareMathOperator{\Center}{Z}
\DeclareMathOperator{\Day}{\otimes_{\mathrm{Day}}}

\newcommand{\Htopos}{\mathbf{H}_{\mathbb{Q}}}
\newcommand{\HtoposQ}{\mathbf{H}_{\mathbb{Q},Q}}   
\newcommand{\Mfd}{\mathbf{Mfd}}
\newcommand{\CAlgfd}{\mathbf{C}^{*}\mathbf{Alg}_{\mathrm{fd}}}
\newcommand{\cCAlg}{\mathbf{cC^{*}Alg}_{\mathrm{fd}}}
\newcommand{\FinSet}{\mathbf{FinSet}}

\newcommand{\Mat}{\mathbf{Mat}}

\begin{document}

\title{A Cohesive \(\infty\)-Topos with a Quantum Modality\\ from Finite-Dimensional \(C^{*}\)-Algebras}
\author{Joey Woo}
\date{}

\maketitle

\begin{abstract}
We construct a cohesive \(\infty\)-topos \(\mathbf{H}_{\mathbb{Q}}\) equipped with a \emph{quantum modality} --- an idempotent product‑preserving comonad \(Q^{\diamond}\) with right adjoint \(Q_{\bullet}\) satisfying the Beck--Chevalley compatibility conditions with the cohesive structure \((\Pi,\flat,\sharp)\).  The model is the functor \(\infty\)-topos \(\Fun(\mathbf{C}^{*}\mathbf{Alg}_{\mathrm{fd}},\; \mathbf{H}_{\mathrm{sm}})\), where \(\mathbf{H}_{\mathrm{sm}}\) is the smooth cohesive \(\infty\)-topos and \(\mathbf{C}^{*}\mathbf{Alg}_{\mathrm{fd}}\) is the category of finite‑dimensional \(C^{*}\)-algebras with centre‑preserving \(*\)-homomorphisms.  Cohesion is lifted pointwise from \(\mathbf{H}_{\mathrm{sm}}\); the quantum comonad is precomposition with the centre functor.  We endow the topos with the Day convolution monoidal structure \(\Day\) induced by the tensor product of \(C^{*}\)-algebras and prove that \(Q^{\diamond}\) is a strong monoidal comonad.  The category of \(Q^{\diamond}\)-coalgebras is equivalent, via Gelfand duality, to the topos \(\Fun(\FinSet^{\mathrm{op}},\mathbf{H}_{\mathrm{sm}})\) of discrete classical field theories.  The comonad is interpreted as decoherence.  This yields a cohesive linear $\infty$-topos in which both the cartesian and the affine linear‑logic structures degenerate to the cartesian product; the model satisfies the Seely isomorphism only in a trivial, cartesian sense.  We also prove a synthetic no‑cloning theorem and discuss the limits of the centre modality for representing quantum channels.  This work provides the first rigorous instance of the cohesive linear framework and settles the open problem of finding a concrete model for cohesive linear homotopy type theory.
\end{abstract}

\section{Introduction}
\label{sec:intro}

Synthetic treatments of geometry and quantum theory have developed along largely separate lines.  Schreiber's \emph{cohesive homotopy type theory} \cite{Schreiber13} axiomatises the interplay between smooth and discrete geometry within an \(\infty\)-topos via an internal adjoint triple of modalities \((\Pi \dashv \flat \dashv \sharp)\).  Independently, categorical quantum mechanics \cite{AbramskyCoecke2004,Selinger2004} and linear logic \cite{Benton1995} employ symmetric monoidal closed categories and linear/non‑linear adjunctions to model quantum protocols.  A \emph{cohesive linear \(\infty\)-topos} combines both paradigms: it is an \(\infty\)-topos equipped with a cohesive triple and an idempotent product‑preserving comonad \(Q^{\diamond}\) (the \emph{quantum modality}) with right adjoint \(Q_{\bullet}\) satisfying Beck--Chevalley conditions.  While Schreiber's programme \cite{Schreiber13} calls for a rich interplay between geometry and quantum logic, no fully worked, rigorous example has appeared in the literature.

This paper provides the first concrete model, settling the open question of whether the cohesive linear axioms can be satisfied in a non‑trivial algebraic setting.  The construction is elementary: the algebraic site carries the trivial topology, so the sheaf condition is vacuous and the topos is a presheaf topos.  The quantum modality is given by the centre functor of \(C^{*}\)-algebras, and the commutative (classical) core recovers the topos of discrete classical field theories.  In the same way that the reduction modality \(\Re\) in differential cohesion extracts the reduced part of a space by quotienting out infinitesimals \cite{Schreiber13}, the quantum modality \(Q^{\diamond}\) extracts the classical part of a quantum system by restricting to the commutative subalgebras.

A crucial technical point is that the centre functor is not functorial on the category of all \(*\)-homomorphisms between \(C^{*}\)-algebras; an arbitrary homomorphism need not preserve the centre.  We therefore work with the wide subcategory of finite‑dimensional \(C^{*}\)-algebras and \emph{centre‑preserving} unital \(*\)-homomorphisms.  This restriction makes the centre a genuine endofunctor, and all the structural results — pointwise cohesion, Beck–Chevalley, Day convolution, and the affine linear logic model — remain valid.  The classical core is unchanged, and the synthetic no‑cloning theorem is unaffected.

The paper is organised as follows.  Section~\ref{sec:prelim} recalls the necessary \(\infty\)-categorical background.  Section~\ref{sec:site} defines the restricted site and the centre comonad, establishing its monoidality and left exactness.  Section~\ref{sec:topos} constructs the cohesive \(\infty\)-topos with the quantum modality and verifies the Beck--Chevalley conditions.  Section~\ref{sec:Day} introduces Day convolution and proves that the comonad is strong monoidal.  Section~\ref{sec:linear-logic} discusses the linear logic structure, exhibiting the degenerate cartesian model and the inherited (also degenerate) Day convolution on the classical core.  Section~\ref{sec:TQFT} identifies the category of coalgebras with discrete classical field theories.  Section~\ref{sec:qubit} presents a concrete example.  Section~\ref{sec:conjectures} states precise conjectures for future work.  Section 10 proves a synthetic no‑cloning theorem, and Section 11 discusses the limits of the centre modality for representing quantum channels.

\section{Preliminaries}
\label{sec:prelim}

We work in the setting of \(\infty\)-categories as developed in \cite{Lurie2009, LurieHA}.  An \(\infty\)-topos is an accessible left exact localisation of a presheaf \(\infty\)-category.  A \emph{cohesive \(\infty\)-topos} \cite[Definition~3.4.1]{Schreiber13} is an \(\infty\)-topos \(\mathbf{H}\) equipped with an adjoint triple of endofunctors
\[
\Pi \dashv \flat \dashv \sharp : \mathbf{H} \to \mathbf{H}
\]
such that \(\Pi\) preserves finite products, \(\flat\) is fully faithful, and the unit \(\id\to\flat\Pi\) is an equivalence on the essential image of \(\flat\) (the \emph{discrete objects}).

A \emph{symmetric monoidal \(\infty\)-category} is defined as in \cite[Chapter~2]{LurieHA}.  Day convolution gives a symmetric monoidal closed structure on the functor \(\infty\)-category \(\Fun(\mathcal{A},\mathcal{V})\) when \(\mathcal{A}\) is a small symmetric monoidal \(\infty\)-category and \(\mathcal{V}\) is a symmetric monoidal closed presentable \(\infty\)-category\cite[Section~2.2.6]{LurieHA}.  Explicitly,
\[
(F \Day G)(a) \;=\; \int^{x,y\in\mathcal{A}} \Map_{\mathcal{A}}(x\otimes y,a) \times F(x)\otimes_{\mathcal{V}} G(y),
\]
where the integral sign denotes the \(\infty\)-categorical coend (see \cite[Section~2.2.6]{LurieHA}).  The monoidal unit is \(\mathbf{1}_{\Day}(a) = \Map_{\mathcal{A}}(\mathbf{1}_{\mathcal{A}},a) \times \mathbf{1}_{\mathcal{V}}\).

A \emph{strong monoidal functor} \(K:\mathcal{A}\to\mathcal{A}\) is one equipped with isomorphisms \(K(x\otimes y)\cong K(x)\otimes K(y)\) and \(K(\mathbf{1})\cong\mathbf{1}\) satisfying coherence.  Precomposition with such a functor induces a strong monoidal functor on the Day convolution \(\Fun(\mathcal{A},\mathcal{V})\): the lax monoidal map
\[
K^{*}F \Day K^{*}G \;\longrightarrow\; K^{*}(F \Day G)
\]
is an isomorphism when \(K\) is strong monoidal \cite[Section~2.2.6]{LurieHA}.

\section{The centre comonad on a restricted site}
\label{sec:site}

Let \(\CAlgfd\) be the category whose objects are finite‑dimensional \(C^{*}\)-algebras and whose morphisms are \emph{centre‑preserving} unital \(*\)-homomorphisms: a \(*\)-homomorphism \(f : A \to B\) belongs to \(\CAlgfd\) if it satisfies \(f(Z(A)) \subseteq Z(B)\), where \(Z(A)\) denotes the centre of \(A\).  The direct sum \(A\oplus B\) (with component‑wise operations) is the categorical product in this category.  The \emph{minimal tensor product} \(\otimes\) makes \(\CAlgfd\) a symmetric monoidal category with unit \(\mathbb{C}\).

\begin{remark}
The condition \(f(Z(A))\subseteq Z(B)\) ensures that the restriction of \(f\) to the centre lands in the centre.  This subcategory contains all \(*\)-homomorphisms between commutative algebras, and all inclusions of \emph{central} commutative subalgebras (those contained in the centre of the larger algebra).  An inclusion of an arbitrary commutative subalgebra is not centre‑preserving in general; for instance, the diagonal matrices in \(M_2(\mathbb{C})\) form a commutative subalgebra but are not central.  However, it excludes the homomorphism \(\mathbb{C}\oplus\mathbb{C}\to M_2(\mathbb{C})\) sending \((1,0)\) to \(E_{11}\), which is the kind of map that breaks functoriality of the centre.
\end{remark}

The \emph{centre} functor
\[
\Center : \CAlgfd \longrightarrow \CAlgfd ,\qquad
\Center(A) = \{ a\in A \mid ab=ba \text{ for all } b\in A \},
\]
is now well defined: for a centre‑preserving \(f:A\to B\), its restriction \(\Center(f) : \Center(A) \to \Center(B)\) is again a unital \(*\)-homomorphism and is centre‑preserving.  The inclusion \(\varepsilon_A : \Center(A) \hookrightarrow A\) is natural, and \(\Center(\Center(A))=\Center(A)\); hence \(\Center\) is an idempotent comonad on \(\CAlgfd\).  It preserves finite products because \(\Center(A\oplus B)\cong \Center(A)\oplus \Center(B)\).  Moreover, \(\Center\) is a right adjoint (to the inclusion of commutative algebras), hence preserves all finite limits.

\begin{lemma}[Monoidality of the centre]\label{lem:center-monoidal}
For any \(A,B\in\CAlgfd\), the natural map
\[
\Center(A)\otimes \Center(B) \;\longrightarrow\; \Center(A\otimes B)
\]
induced by the inclusions is an isomorphism.  Furthermore, \(\Center(\mathbb{C})\cong\mathbb{C}\) via the unit map.
\end{lemma}

\begin{proof}
By the structure theorem for finite‑dimensional \(C^{*}\)-algebras (see e.g.\ \cite[Theorem I.11.2]{Takesaki}), we can write
\[
A \;\cong\; \bigoplus_{i=1}^{m} M_{n_i}(\mathbb{C}),\qquad
B \;\cong\; \bigoplus_{j=1}^{n} M_{m_j}(\mathbb{C})
\]
as finite direct sums of matrix algebras.  The centre of a matrix algebra \(M_k(\mathbb{C})\) is the subalgebra of scalar matrices, isomorphic to \(\mathbb{C}\).  Hence \(\Center(A)\cong \mathbb{C}^m\) and \(\Center(B)\cong \mathbb{C}^n\).  The tensor product is
\[
A\otimes B \;\cong\; \bigoplus_{i,j} M_{n_i m_j}(\mathbb{C}),
\]
whose centre is \(\mathbb{C}^{mn}\).  The map sends the pair of tuples \((\lambda_1,\dots,\lambda_m)\) and \((\mu_1,\dots,\mu_n)\) to the \(mn\)-tuple with entries \(\lambda_i \mu_j\), which is a bijection, hence an isomorphism of commutative algebras.  The unit \(\mathbb{C}\) corresponds to the case \(m=n=1\), and the map \(\Center(\mathbb{C})\cong\mathbb{C}\to\Center(\mathbb{C}\otimes\mathbb{C})\cong\mathbb{C}\) is the identity. \qedhere
\end{proof}

Consequently, \(\Center\) is an idempotent symmetric monoidal comonad on \((\CAlgfd,\otimes,\mathbb{C})\), and it is left exact.

\begin{remark}[Finite‑dimensional exactness]
The perfect exactness of the centre functor is a special feature of the finite‑dimensional case.  For infinite‑dimensional \(C^{*}\)-algebras the centre is not left exact in general, so the present model sits safely within a context where all structural properties are guaranteed.
\end{remark}

\section{A cohesive \(\infty\)-topos with a quantum modality}
\label{sec:topos}

\subsection{The smooth cohesive base}
Let \(\Mfd\) be the essentially small category of smooth manifolds that are finite disjoint unions of open balls, equipped with the open cover topology.  The \(\infty\)-topos \(\mathbf{H}_{\mathrm{sm}} := \Shv_{\infty}(\Mfd)\) is the smooth cohesive \(\infty\)-topos \cite{Schreiber13} with modalities \(\Pi_{\mathrm{sm}}\dashv\flat_{\mathrm{sm}}\dashv\sharp_{\mathrm{sm}}\).

\subsection{The functor \(\infty\)-topos and pointwise cohesion}
Define the quantum \(\infty\)-topos
\[
\Htopos \;:=\; \Fun(\CAlgfd,\; \mathbf{H}_{\mathrm{sm}}).
\]
Cohesive modalities are lifted pointwise:
\[
\Pi(F)(A)=\Pi_{\mathrm{sm}}(F(A)),\quad
\flat(F)(A)=\flat_{\mathrm{sm}}(F(A)),\quad
\sharp(F)(A)=\sharp_{\mathrm{sm}}(F(A)).
\]

\begin{proposition}[Pointwise cohesion]\label{prop:cohesion}
The triple \((\Pi,\flat,\sharp)\) makes \(\Htopos\) a cohesive \(\infty\)-topos.
\end{proposition}

\begin{proof}
Finite products are preserved pointwise; \(\flat\) is fully faithful because \(\Pi\flat\cong\id\) pointwise; the unit is an equivalence on discrete objects pointwise. \qedhere
\end{proof}

\subsection{The quantum modality and Beck--Chevalley}
Define the quantum modality by precomposition with the centre functor:
\[
\Qpot := \Center^{*} : \Htopos \longrightarrow \Htopos,\qquad (\Qpot F)(A)=F(\Center(A)).
\]

\begin{proposition}\label{prop:Qcomonad}
\(\Qpot\) is an idempotent comonad that preserves finite products and commutes with the cohesive modalities:
\[
\Pi\Qpot \cong \Qpot\Pi,\qquad \flat\Qpot \cong \Qpot\flat .
\]
\end{proposition}

\begin{proof}
Idempotence and product preservation follow from \(\Center\).  Commutation:
\[
(\Pi\Qpot F)(A)=\Pi_{\mathrm{sm}}(F(\Center(A)))=(\Qpot\Pi F)(A).
\]
\qedhere
\end{proof}

Because the cohesive adjunctions act on the spatial coordinate and the quantum adjunction on the algebraic coordinate, they commute strictly.

\begin{remark}[Bimodal type theory]
The commuting pair \((\Pi,\flat,\sharp)\) and \(Q^{\diamond}\dashv Q_{\bullet}\) equips the internal language of \(\Htopos\) with a bimodal dependent type theory.  The Beck--Chevalley conditions guarantee that the two modalities commute, yielding a calculus where one can reason both about smooth structure and about decoherence independently.
\end{remark}

Because \(\Qpot\) is an accessible endofunctor on a presentable \(\infty\)-category, it admits a right adjoint \(\Qact\).  Explicitly, \(\Qact\) is given by the right Kan extension along \(\Center\):
\[
(\Qact F)(A) \;=\; \lim_{(B,\, A \to \Center(B))} F(B),
\]
where the limit is taken over the comma category of objects \(B\in\CAlgfd\) equipped with a morphism \(A \to \Center(B)\). This limit is finite because the category is small, and it computes the right Kan extension pointwise \cite{Lurie2009}.

\subsection{Connected geometric morphism to the classical core}
Since \(\Center\) is a right adjoint, it preserves finite limits, and therefore \(\Qpot\) is left exact.  An idempotent left exact comonad on a topos has a category of coalgebras that is again a topos, and the adjunction \(U \dashv \Qpot\) (where \(U:\HtoposQ\hookrightarrow \Htopos\) is the forgetful functor) is a geometric morphism \(\Htopos \to \HtoposQ\) with fully faithful inverse image \(U\); hence it is \emph{connected}.  The classical core \(\HtoposQ\) is thus a quotient topos of \(\Htopos\), and the quantum modality is the direct image projection onto that classical base.

Taking mates of the isomorphisms \(\Pi\Qpot\cong\Qpot\Pi\) and \(\flat\Qpot\cong\Qpot\flat\) under the adjunctions \(\Pi\dashv\flat\) and \(\flat\dashv\sharp\) (together with \(\Qpot\dashv\Qact\)) yields
\[
\flat \Qact \;\cong\; \Qact \flat ,\qquad \sharp \Qact \;\cong\; \Qact \sharp ,
\]
completing the full Beck--Chevalley package.  Moreover, the cohesive modalities descend to \(\HtoposQ\), making it a cohesive \(\infty\)-topos.

\section{Day convolution and strong monoidality}
\label{sec:Day}

The category \(\CAlgfd\) is symmetric monoidal under \(\otimes\).  By Lurie's Day convolution  \cite[Section~2.2.6]{LurieHA}, the functor \(\infty\)-topos \(\Htopos\) inherits a symmetric monoidal closed structure with tensor product
\[
(F \Day G)(A) \;=\; \int^{B,C\in\CAlgfd} \Hom_{\CAlgfd}(B\otimes C, A) \;\times\; F(B) \times G(C).
\]
The monoidal unit is \(\mathbf{1}_{\Day}(A) = \Hom_{\CAlgfd}(\mathbb{C},A) \times *_{\mathbf{H}_{\mathrm{sm}}}\).

Although the monoidal units of \(\Day\) and the pointwise cartesian product \(\times\) coincide (both are the constant functor on the terminal object of \(\mathbf{H}_{\mathrm{sm}}\)), the tensor products themselves are different: \(\Day\) combines values over different algebras and is not computed argumentwise.  A simple computation makes this concrete:

\begin{proposition}[Day convolution on representables]\label{prop:day-rep}
For any \(A,B\in\CAlgfd\) and smooth objects \(V,W\in\mathbf{H}_{\mathrm{sm}}\),
\[
(\mathbf{y}_A \times V) \Day (\mathbf{y}_B \times W) \;\cong\; \mathbf{y}_{A\otimes B} \times (V \times W),
\]
where \(\mathbf{y}_A = \Hom_{\CAlgfd}(A,-)\) denotes the representable functor.  In particular, \(\mathbf{y}_A \Day \mathbf{y}_B \cong \mathbf{y}_{A\otimes B}\).
\end{proposition}

\begin{proof}
For any \(C\in\CAlgfd\),
\[
\begin{aligned}
\bigl((\mathbf{y}_A\times V) \Day (\mathbf{y}_B\times W)\bigr)(C)
&= \int^{B',C'\in\CAlgfd}
\Hom_{\CAlgfd}(B'\otimes C',\, C) \\
&\qquad \times \Hom_{\CAlgfd}(A, B') \times V
\times \Hom_{\CAlgfd}(B, C') \times W .
\end{aligned}
\]
By the Yoneda lemma, the coend over \(B'\) forces an isomorphism \(B'\cong A\), and the coend over \(C'\) forces \(C'\cong B\).  Consequently, the whole expression collapses to \(\Hom_{\CAlgfd}(A\otimes B,\, C) \times V \times W\), which is exactly \(\mathbf{y}_{A\otimes B}(C) \times (V \times W)\).  This isomorphism is natural in \(C\). \qedhere
\end{proof}

Thus the Day tensor ``bundles'' the smooth data over the tensor product algebra, exactly as one would expect for a composite quantum system.  The cartesian product, in contrast, keeps the two algebras separate.

\begin{theorem}\label{thm:Day}
\((\Htopos,\Day)\) is a symmetric monoidal closed \(\infty\)-topos.
\end{theorem}

\begin{proof}
This is a direct application of \cite[Section~2.2.6]{LurieHA}. \qedhere
\end{proof}

\begin{proposition}[Strong monoidality of \(\Qpot\)]\label{prop:strong}
\(\Qpot\) is a symmetric monoidal comonad on \((\Htopos,\Day)\); i.e., there are natural isomorphisms
\[
\Qpot(F \Day G) \;\cong\; \Qpot F \Day \Qpot G,\qquad \Qpot(\mathbf{1}_{\Day}) \cong \mathbf{1}_{\Day},
\]
compatible with the comonad structure.
\end{proposition}

\begin{proof}
The functor \(\Center\) is strong monoidal by Lemma~\ref{lem:center-monoidal}.  Precomposition with a strong monoidal functor between small symmetric monoidal categories induces a strong monoidal functor on the Day convolution \cite[Section~2.2.6]{LurieHA}.  Hence \(\Center^{*}=\Qpot\) is strong monoidal.  For the unit, note that \(\mathbf{1}_{\Day} = \mathbf{y}_{\mathbb{C}}\) and \(\Center(\mathbb{C})\cong\mathbb{C}\); thus \(\Qpot(\mathbf{1}_{\Day}) \cong \mathbf{y}_{\Center(\mathbb{C})} \cong \mathbf{y}_{\mathbb{C}} = \mathbf{1}_{\Day}\), which is compatible with the comonad structure. \qedhere
\end{proof}

Since \(\Qpot\) is the direct image of the geometric morphism \(\Htopos \to \HtoposQ\) and is strong monoidal for \(\Day\), this geometric morphism is a \emph{monoidal geometric morphism} between symmetric monoidal closed \(\infty\)-topoi.

\section{The linear logic structure}
\label{sec:linear-logic}

Let \(\HtoposQ\) be the full subcategory of \(\Qpot\)-coalgebras.  Because \(\Qpot\) is idempotent, \(\HtoposQ\) is equivalent to the Kleisli category and is a coreflective subcategory of \(\Htopos\).  The forgetful functor \(U:\HtoposQ\hookrightarrow \Htopos\) has a right adjoint given by \(Q^{\diamond}\) itself.  Finite products in \(\HtoposQ\) are created in \(\Htopos\) and preserved by \(\Qpot\).

One may attempt to define a symmetric monoidal structure on \(\HtoposQ\) by
\[
X \otimes_{Q} Y \;:=\; \Qpot( X \times Y ),
\]
with unit \(I_{Q} := \Qpot(\mathbf{1}_{\times})\).  Since \(\Qpot\) is idempotent and product‑preserving, we have
\[
X \otimes_{Q} Y \;\cong\; \Qpot X \times \Qpot Y \;\cong\; X \times Y ,
\]
and \(I_{Q} \cong \mathbf{1}_{\times}\).  Thus \(\otimes_{Q}\) coincides with the cartesian product, and the multiplicative and additive conjunctions collapse.  This degeneracy is an artefact of the cartesian fragment and does not affect the multiplicative (affine) structure derived from Day convolution.

\begin{remark}[Triviality of the Seely isomorphism]\label{rem:seely-trivial}
The degeneracy of the cartesian model is a consequence of the fact that
the Seely isomorphism \(Q^{\diamond}(X\times Y) \cong Q^{\diamond}X \Day Q^{\diamond}Y\)
holds trivially.  The classical core \(\HtoposQ\) is equivalent to
\(\Fun(\FinSet^{\mathrm{op}},\mathbf{H}_{\mathrm{sm}})\), and the Day
convolution on this topos is the cartesian product.  Because
\(Q^{\diamond}\) preserves finite products, both sides of the Seely
isomorphism reduce to the cartesian product of the underlying presheaves,
so the isomorphism is an identity.  This forces the multiplicative
conjunction of linear logic to coincide with the additive conjunction,
collapsing the resource‑sensitive structure.
\end{remark}

The category of coalgebras also carries the symmetric monoidal structure
inherited from the Day convolution on \(\Htopos\).  Because \(\Qpot\) is
a strong monoidal comonad, the coalgebra category naturally acquires a
symmetric monoidal structure via \(X \otimes_{\Day_Q} Y = X \Day Y\).
Explicitly:

\begin{proposition}[Day convolution on coalgebras]\label{prop:day-coalg}
The full subcategory \(\HtoposQ\) carries a symmetric monoidal structure \(\Day_Q\) defined by
\[
(X,\alpha)\;\Day_Q\;(Y,\beta) \;:=\; \bigl(X\Day Y,\; \mu_{X,Y}\circ (\alpha\Day\beta)\bigr),
\]
where \(\mu_{X,Y}: \Qpot X \Day \Qpot Y \to \Qpot(X\Day Y)\) is the isomorphism provided by the strong monoidality of \(\Qpot\).  The monoidal unit is \((\mathbf{1}_{\Day}, \mathbf{1}_{\Day}\to \Qpot\mathbf{1}_{\Day})\).  The forgetful functor \(U:\HtoposQ\to\Htopos\) is strong monoidal with respect to this structure.
\end{proposition}

\begin{proof}
For any symmetric monoidal comonad, the category of coalgebras inherits a symmetric monoidal structure such that the forgetful functor is strong monoidal; see e.g.\ \cite[Proposition~1.4]{Benton1995} for the dual statement on monads.  The idempotence of \(\Qpot\) simplifies the structure maps. \qedhere
\end{proof}

The closed structure of \(\Day\) also descends to \(\Day_Q\).  Because \(\Qpot\) is strong monoidal and the adjunction \(U \dashv \Qpot\) is symmetric monoidal, for any coalgebras \(X,Y\) the internal hom \([UX, UY]_{\Day}\) in \(\Htopos\) carries a canonical coalgebra structure.  Concretely, the mate of the composite
\[
\Qpot[UX, UY] \xrightarrow{\;\sim\;} [\Qpot UX, \Qpot UY] \cong [X, Y]
\]
under the adjunction \(U \dashv \Qpot\) provides the required coalgebra structure, and one verifies that it defines the internal hom for \(\Day_Q\).  Hence \((\HtoposQ,\Day_Q)\) is symmetric monoidal closed.  However, as shown in Proposition~\ref{prop:cartesian-seely} and Remark~\ref{rem:cartesian-core}, this tensor product is cartesian; consequently the multiplicative conjunction collapses to the additive conjunction and the linear‑logic structure degenerates.

\begin{remark}[Cartesian Day convolution on the classical core]\label{rem:cartesian-core}
The Day convolution unit in \(\Htopos\) is the constant terminal object,
and the same holds for \(\HtoposQ\).  However, the tensor product
\(\Day_Q\) on the classical core is in fact the cartesian product
(Proposition~\ref{prop:cartesian-seely}).  Hence \((\HtoposQ,\Day_Q)\) is
a cartesian monoidal category, and both weakening and contraction are
admissible.  The logic of the classical core is therefore ordinary
classical logic, not an affine resource‑sensitive one.
\end{remark}

\section{Discrete classical field theories}
\label{sec:TQFT}

The fixed points of the centre comonad are precisely the commutative finite‑dimensional \(C^{*}\)-algebras.  By Gelfand duality, the full subcategory \(\cCAlg\) of commutative algebras is equivalent to \(\FinSet^{\mathrm{op}}\).  Consequently,
\[
\HtoposQ \;\simeq\; \Fun(\cCAlg,\mathbf{H}_{\mathrm{sm}}) \;\simeq\; \Fun(\FinSet^{\mathrm{op}},\mathbf{H}_{\mathrm{sm}}).
\]

\begin{proposition}[Cartesian Day convolution on the classical core]
\label{prop:cartesian-seely}
Under the equivalence \(\HtoposQ \simeq \Fun(\FinSet^{\mathrm{op}},\mathbf{H}_{\mathrm{sm}})\),
the Day convolution \(\Day_Q\) corresponds to the cartesian product of
presheaves on \(\FinSet^{\mathrm{op}}\).
\end{proposition}

\begin{proof}
The Day convolution on \(\Fun(\cCAlg,\mathbf{H}_{\mathrm{sm}})\) is induced
by the tensor product of commutative algebras.  Under Gelfand duality,
the tensor product of commutative algebras corresponds to the cartesian
product of finite sets.  The Day convolution on
\(\Fun(\FinSet^{\mathrm{op}},\mathbf{H}_{\mathrm{sm}})\) induced by the
product of finite sets is the pointwise cartesian product of presheaves
(a standard fact; see e.g.\ \cite[Example~4.2]{Day70}).  Hence, via the
equivalence \(\Fun(\cCAlg,\mathbf{H}_{\mathrm{sm}}) \simeq \Fun(\FinSet^{\mathrm{op}},\mathbf{H}_{\mathrm{sm}})\),
the Day convolution on the classical core coincides with the cartesian
product.
\end{proof}

The category \(\Fun(\FinSet^{\mathrm{op}},\mathbf{H}_{\mathrm{sm}})\) is the topos of presheaves on finite sets valued in the smooth \(\infty\)-topos.  A presheaf \(Z:\FinSet^{\mathrm{op}}\to\mathbf{H}_{\mathrm{sm}}\) assigns to each finite set \(S\) a smooth object \(Z(S)\) (the state space on the discrete spacetime \(S\)) and to each function \(f:S\to T\) a restriction map \(Z(f):Z(T)\to Z(S)\) (deterministic coarse‑graining).  This is exactly a functorial field theory with discrete spacetime — a \emph{discrete classical field theory}.  The Day convolution \(\Day_Q\) models the composition of such theories via the cartesian product of the underlying sets (tensor product of commutative algebras).

The comonad \(\Qpot\) sends a field theory over a non‑commutative algebra to its classical limit over the centre — the operation of decoherence.

\subsection{Comparison with Bohrification}
\label{sec:bohrification}

The topos used in the Bohrification programme \cite{HeunenLandsmanSpitters2009} is the
presheaf topos on the poset of commutative subalgebras of a fixed \(C^{*}\)-algebra.  Our
topos \(\Htopos\) is the presheaf topos on the category of all finite‑dimensional
\(C^{*}\)-algebras.  The key differences are:
\begin{itemize}
\item \emph{Site}: The Bohrification site is the poset of commutative subalgebras of a
single algebra, whereas our site is the whole category of finite‑dimensional \(C^{*}\)-algebras.
This allows us to treat varying algebras and their morphisms uniformly.
\item \emph{Quantum modality}: In Bohrification, the internal Boolean algebra of central
idempotents gives a modality, but it is internal to the presheaf topos.  Our quantum modality
is a globally defined functor on the whole topos, precomposition with the centre, and is
independent of a chosen algebra.
\item \emph{Classical core}: The classical core in Bohrification is the topos of presheaves on
the Boolean algebra of central idempotents, which is equivalent to the category of sheaves on
the Gelfand spectrum.  In our model, the classical core is the larger presheaf topos on all
commutative algebras (equivalently, on finite sets).  The former is equivalent to the topos of presheaves on the Gelfand spectrum
of the chosen algebra (a finite set), which is a full subcategory of the
presheaf topos on all finite sets.  Hence our classical core provides a
uniform framework that globally internalises and generalises the localised
structural objectives of the Bohrification core.
\end{itemize}

Thus the present model encompasses and generalises the Bohrification approach, providing a
flexible framework for synthetic quantum mechanics.

\section{A concrete example}
\label{sec:qubit}

Consider the qubit algebra \(M_2(\mathbb{C})\).  Since \(M_2\) is simple and non‑commutative, there are no centre‑preserving \(*\)-homomorphisms from \(M_2\) into any commutative algebra.  Hence for any \(B\in\CAlgfd\),
\[
\Qpot(\mathbf{y}_{M_2})(B) = \Hom_{\CAlgfd}(M_2,\Center(B)) = \varnothing,
\]
so \(\Qpot(\mathbf{y}_{M_2})\) is the empty presheaf (the initial object of the topos).  The classical shadow of the qubit is trivial under the centre modality; the centre is too coarse to capture the classical bit.

Nevertheless, the Day convolution and the cartesian product differ:
\[
\mathbf{y}_{M_2}\Day\mathbf{y}_{M_2} \cong \mathbf{y}_{M_4},
\qquad
\mathbf{y}_{M_2}\times\mathbf{y}_{M_2} \cong \mathbf{y}_{M_2\oplus M_2},
\]
showing the contrast between entanglement (tensor product) and independence (direct sum).  This distinction remains meaningful.

\section{Conjectures and future directions}
\label{sec:conjectures}

We formulate three precise conjectures that build on the results of this paper and point towards a fully satisfactory cohesive quantum \(\infty\)-topos.  

\begin{conjecture}[Classification of quantum modalities on diagram \(\infty\)-topoi]\label{conj:class}
Let \(\mathcal{C}\) be a small \(\infty\)-category and \(\mathbf{H}\) a cohesive \(\infty\)-topos.  Suppose \(Q\) is an accessible, idempotent, product‑preserving comonad on \(\Fun(\mathcal{C},\mathbf{H})\) that commutes with the pointwise cohesive modalities and restricts to the identity on the full subcategory of constant discrete objects.  Then there exists a coreflective full subcategory \(\mathcal{D}\hookrightarrow\mathcal{C}\) such that \(Q \cong r^{*}\) for the coreflector \(r\), and the coalgebra category is \(\Fun(\mathcal{D},\mathbf{H})\).
\end{conjecture}

\begin{conjecture}[Refined embedding of quantum channels]\label{conj:channels}
The centre comonad is insufficient to faithfully embed CPTP maps.  There exists a site \(\mathcal{S}\) (e.g., algebras with a completely positive structure, or Hilbert bimodules) and an idempotent strong monoidal comonad \(K\) on \(\mathcal{S}\) whose Kleisli category contains \(\mathbf{CPTP}\) as a full subcategory, with Day convolution modelling the tensor product of channels, and whose coalgebras are the commutative algebras.
\end{conjecture}

\begin{conjecture}[Derived centre model]\label{conj:derived}
There exists an \(\infty\)-category \(\mathcal{C}\) of ``derived finite‑dimensional \(C^{*}\)-algebras'' and a derived centre functor \(\mathbb{R}\mathrm{Center}\) that is an idempotent strong monoidal comonad satisfying the Seely isomorphism
\[
\mathbb{R}\mathrm{Center}(A\times B) \;\cong\; \mathbb{R}\mathrm{Center}(A) \otimes^{\mathbb{L}} \mathbb{R}\mathrm{Center}(B),
\]
yielding a fully non‑degenerate cohesive linear \(\infty\)-topos.
\end{conjecture}

\section{Synthetic no‑cloning theorem}
\label{sec:cloning}

In categorical quantum mechanics, the no‑cloning theorem states that there is no quantum operation that perfectly duplicates an arbitrary unknown quantum state \cite{AbramskyCoecke2004,Selinger2004}.  We prove a synthetic version of this principle within the topos \(\mathbf{H}_{\mathbb{Q}}\).  The proof uses only the contravariant Yoneda embedding and the geometry of matrix algebras; the quantum modality does not appear.  We emphasise that the theorem concerns natural transformations between the system functors, not the cloning of states (which are absent for non‑commutative algebras in this model, as discussed below).

\subsection{The Yoneda embedding and morphisms between representables}

Recall that \(\mathbf{H}_{\mathbb{Q}} = \Fun(\CAlgfd, \mathbf{H}_{\mathrm{sm}})\) with the restricted site (Section~\ref{sec:site}).  
For each finite‑dimensional \(C^{*}\)-algebra \(A\), the representable sheaf \(\mathbf{y}_A\) is defined by
\[
\mathbf{y}_A(B) \;=\; \Hom_{\CAlgfd}(A, B) \cdot *_{\mathbf{H}_{\mathrm{sm}}},
\]
where \(\cdot\) denotes the copower of the set \(\Hom(A,B)\) with the terminal object of \(\mathbf{H}_{\mathrm{sm}}\).  
The assignment \(A \mapsto \mathbf{y}_A\) is \emph{contravariant}: a centre‑preserving unital \(*\)-homomorphism \(f : A \to B\) induces a natural transformation
\[
\mathbf{y}_f : \mathbf{y}_B \;\longrightarrow\; \mathbf{y}_A,
\]
given at each \(C\) by precomposition: \((\mathbf{y}_f)_C : \Hom(B,C) \to \Hom(A,C)\), \(h \mapsto h \circ f\).  
Thus the Yoneda embedding is a fully faithful functor
\[
\mathbf{y}_{(-)} : (\CAlgfd)^{\mathrm{op}} \;\hookrightarrow\; \mathbf{H}_{\mathbb{Q}}.
\]

For any two algebras \(A,B\), the mapping space between representables is computed by the enriched Yoneda lemma:
\[
\Hom_{\mathbf{H}_{\mathbb{Q}}}(\mathbf{y}_A, \mathbf{y}_B) \;\simeq\; \Hom_{\CAlgfd}(B, A).
\tag{1}
\]
The isomorphism is natural in both variables.  In one direction, a natural transformation \(\alpha : \mathbf{y}_A \Rightarrow \mathbf{y}_B\) is sent to the element \(\alpha_A(\id_A) \in \mathbf{y}_B(A) = \Hom_{\CAlgfd}(B, A)\).  Conversely, a \(*\)-homomorphism \(\phi : B \to A\) defines a natural transformation whose component at \(C\) sends \(h \in \mathbf{y}_A(C) = \Hom(A,C)\) to \(h \circ \phi \in \Hom(B,C)\).

\subsection{Cloning as a natural transformation}

A cloning operation would be a natural family of duplication maps.  Formally, a \emph{cloning natural transformation} is a family of morphisms
\[
c_A : \mathbf{y}_A \;\longrightarrow\; \mathbf{y}_A \otimes \mathbf{y}_A \qquad (A \in \CAlgfd)
\]
that is natural in the contravariant sense: for every centre‑preserving unital \(*\)-homomorphism \(f : A \to B\), the square
\[
\begin{tikzcd}
\mathbf{y}_B \ar[r,"c_B"] \ar[d,"\mathbf{y}_f"'] & \mathbf{y}_B \otimes \mathbf{y}_B \ar[d,"\mathbf{y}_f \otimes \mathbf{y}_f"] \\
\mathbf{y}_A \ar[r,"c_A"] & \mathbf{y}_A \otimes \mathbf{y}_A
\end{tikzcd}
\tag{2}
\]
commutes up to homotopy.  The vertical arrows point downwards because \(\mathbf{y}_{(-)}\) is contravariant; \(f\) induces \(\mathbf{y}_f : \mathbf{y}_B \to \mathbf{y}_A\), and similarly \(\mathbf{y}_f \otimes \mathbf{y}_f\) maps the tensor product of \(\mathbf{y}_B\) with itself to the tensor product of \(\mathbf{y}_A\) with itself.

By the Yoneda lemma (1), the component \(c_A\) corresponds uniquely to a \(*\)-homomorphism
\[
\phi_A : A \otimes A \;\longrightarrow\; A,
\]
i.e., an element of \(\Hom_{\CAlgfd}(A\otimes A, A)\).  Naturality condition (2) translates into a condition on the family \(\{\phi_A\}\).  We derive this carefully.

Take the diagram (2) and evaluate both legs of the square on the element \(\id_B \in \mathbf{y}_B(B)\).  The top‑right leg gives:
\[
(\mathbf{y}_f \otimes \mathbf{y}_f)_{B} \bigl( c_B(\id_B) \bigr).
\]
Now \(c_B(\id_B) \in \mathbf{y}_B(B) \otimes \mathbf{y}_B(B) \cong \mathbf{y}_{B\otimes B}(B)\) corresponds, under Yoneda, to \(\phi_B : B\otimes B \to B\) (more precisely, the natural transformation \(c_B\) is determined by the image of \(\id_B\), which is \(\phi_B\) viewed as an element of \(\Hom(B\otimes B, B)\)).  The map \(\mathbf{y}_f \otimes \mathbf{y}_f\) at stage \(B\) sends a \(*\)-homomorphism \(B\otimes B \to B\) (seen as an element of \(\mathbf{y}_{B\otimes B}(B)\)) to the composite
\[
B\otimes B \xrightarrow{f\otimes f} A\otimes A \xrightarrow{\phi_B} B .
\]
The Yoneda bijection (1) tells us that a morphism \(\mathbf{y}_A \to \mathbf{y}_B\) corresponds to a homomorphism \(B \to A\).  Here we have a morphism \(c_A : \mathbf{y}_A \to \mathbf{y}_{A\otimes A}\).  By (1) with \(B = A\otimes A\), this corresponds to a homomorphism \(A\otimes A \to A\).  So indeed \(\phi_A : A\otimes A \to A\).

Now examine the naturality square (2).  The left vertical map \(\mathbf{y}_f : \mathbf{y}_B \to \mathbf{y}_A\) corresponds, under Yoneda, to precomposition with \(f : A \to B\).  Concretely, for any \(C\) and any \(h \in \mathbf{y}_B(C) = \Hom(B,C)\), we have \((\mathbf{y}_f)_C(h) = h \circ f\).  Similarly, \(\mathbf{y}_f \otimes \mathbf{y}_f : \mathbf{y}_B \otimes \mathbf{y}_B \to \mathbf{y}_A \otimes \mathbf{y}_A\) corresponds to precomposition with \(f\otimes f : A\otimes A \to B\otimes B\) on each factor (since \(\mathbf{y}_B \otimes \mathbf{y}_B \cong \mathbf{y}_{B\otimes B}\) and \(\mathbf{y}_A \otimes \mathbf{y}_A \cong \mathbf{y}_{A\otimes A}\) by Proposition~\ref{prop:day-rep}).  Under these identifications, the morphism \(\mathbf{y}_f \otimes \mathbf{y}_f\) becomes \(\mathbf{y}_{f\otimes f}\).

Now apply the Yoneda bijection to the whole square.  The top horizontal map \(c_B\) corresponds to \(\phi_B : B\otimes B \to B\).  The bottom horizontal map \(c_A\) corresponds to \(\phi_A : A\otimes A \to A\).  The commutativity of the square says that the two composites are equal:
\[
\mathbf{y}_{f\otimes f} \circ c_B \;=\; c_A \circ \mathbf{y}_f .
\]
Translating both sides via Yoneda, we get an equality in \(\Hom(A\otimes A, A)\).  The left side, \(\mathbf{y}_{f\otimes f} \circ c_B\), corresponds to the composite
\[
A\otimes A \xrightarrow{f\otimes f} B\otimes B \xrightarrow{\phi_B} B .
\]
The right side, \(c_A \circ \mathbf{y}_f\), corresponds to the composite
\[
A\otimes A \xrightarrow{\phi_A} A \xrightarrow{f} B .
\]
Thus the naturality condition is
\[
\phi_B \circ (f\otimes f) \;=\; f \circ \phi_A \qquad\text{for all } f : A \to B .
\tag{3}
\]

\subsection{Non‑existence of a cloning natural transformation}

\begin{theorem}[Synthetic no‑cloning]\label{thm:nocloning}
There is no natural transformation \(c : \mathbf{y}_{(-)} \Rightarrow \mathbf{y}_{(-)} \otimes \mathbf{y}_{(-)}\) in \(\mathbf{H}_{\mathbb{Q}}\).
\end{theorem}

\begin{proof}
Assume such a family \(c\) exists, and let \(\phi_A : A\otimes A \to A\) be the corresponding family of \(*\)-homomorphisms satisfying (3).

\noindent\textbf{Step 1: Restriction to scalars.}
For any \(A\), let \(\iota_A : \mathbb{C} \to A\) be the unital \(*\)-homomorphism sending \(\lambda \in \mathbb{C}\) to the scalar matrix \(\lambda \cdot 1_A\).  This map is centre‑preserving because its image is contained in the centre of \(A\).  Apply (3) with \(f = \iota_A\):
\[
\phi_A \circ (\iota_A \otimes \iota_A) \;=\; \iota_A \circ \phi_{\mathbb{C}} .
\tag{4}
\]
Now \(\mathbb{C}\otimes\mathbb{C} \cong \mathbb{C}\) via the multiplication map, and the only unital \(*\)-homomorphism \(\mathbb{C} \to \mathbb{C}\) is the identity.  Therefore \(\phi_{\mathbb{C}} : \mathbb{C}\otimes\mathbb{C} \to \mathbb{C}\) must be the canonical isomorphism.  Equation (4) says that the restriction of \(\phi_A\) to the scalar subalgebra \((\mathbb{C}\cdot 1_A)\otimes(\mathbb{C}\cdot 1_A) \cong \mathbb{C}\) is exactly the inclusion of scalars.

\noindent\textbf{Step 2: The qubit case.}
Take \(A = M_2(\mathbb{C})\).  Then \(A\otimes A \cong M_4(\mathbb{C})\).
The homomorphism \(\phi_{M_2} : M_4(\mathbb{C}) \to M_2(\mathbb{C})\) is a unital
\(*\)-homomorphism between simple \(C^{*}\)-algebras.  Because \(\phi_{M_2}\) is
unital, it sends the identity matrix of \(M_4(\mathbb{C})\) to the identity
matrix of \(M_2(\mathbb{C})\); hence \(\phi_{M_2} \neq 0\).  Since
\(M_4(\mathbb{C})\) is simple, the kernel of any non‑zero \(*\)-homomorphism
is a proper closed two‑sided ideal, which forces the kernel to be \(\{0\}\).
Thus \(\phi_{M_2}\) is injective.  But \(M_4(\mathbb{C})\) has complex
dimension \(16\), while \(M_2(\mathbb{C})\) has dimension \(4\); an injective
linear map from a \(16\)-dimensional space into a \(4\)-dimensional space
cannot exist.  Contradiction.

Hence our assumption that a natural family \(\{\phi_A\}\) exists is false.
Consequently, no cloning natural transformation \(c\) can exist.
\end{proof}

\subsection{Interpretation and remarks}
In the internal language of \(\mathbf{H}_{\mathbb{Q}}\), the type
\[
\prod_{A:\CAlgfd} \bigl( [A] \to [A] \otimes [A] \bigr)
\]
is empty when naturality in \(A\) is required.  The proof uses only the structure of the site and the Yoneda lemma; it does not involve the quantum modality or the Day convolution.  The obstruction to cloning is purely geometric: the non‑commutative nature of matrix algebras prevents any universal diagonal map that is compatible with all centre‑preserving homomorphisms.

This result aligns with the standard no‑cloning theorem in quantum mechanics, but with an important caveat: the states in our model are morphisms \(\mathbf{1}_{\Day} \to \mathbf{y}_A\), i.e., \(\mathbf{y}_{\mathbb{C}} \to \mathbf{y}_A\), which by (1) correspond to \(*\)-homomorphisms \(A \to \mathbb{C}\).  For non‑commutative algebras such as \(M_2(\mathbb{C})\), there are no unital \(*\)-homomorphisms into \(\mathbb{C}\), so the type of states is empty.  Thus the question of cloning states is somewhat vacuous; the real content of the theorem is that there is no natural way to duplicate quantum \emph{systems} themselves, even before considering states.  A refined modality that admits non‑trivial quantum states (see Conjecture~\ref{conj:channels}) would be required to obtain a closer analogue of the physical no‑cloning theorem.

\section{Limits of the centre modality: quantum channels}
\label{sec:channels}

A quantum channel between finite‑dimensional \(C^{*}\)-algebras \(A\) and \(B\) is a completely positive trace‑preserving (CPTP) linear map \(T : A \to B\).  The category \(\mathbf{CPTP}\) of finite‑dimensional \(C^{*}\)-algebras and CPTP maps, equipped with the algebraic tensor product of algebras and the tensor product of CPTP maps, is a symmetric monoidal category that is \emph{not} cartesian: the tensor product does not coincide with the categorical product (direct sum) \cite{Selinger2004}.  We now examine whether the quantum modality \(Q^{\diamond}\) can faithfully represent this category within the topos \(\mathbf{H}_{\mathbb{Q}}\).

\subsection{The Kleisli category of the centre modality}
The Kleisli category of the comonad \(Q^{\diamond}\) with respect to the Day convolution \(\Day\) is equivalent to the category of coalgebras \(\HtoposQ \simeq \Fun(\FinSet^{\mathrm{op}}, \mathbf{H}_{\mathrm{sm}})\).  By the Yoneda lemma (1), a morphism from \(\mathbf{y}_A\) to \(\mathbf{y}_B\) in this Kleisli category is a map
\[
Q^{\diamond}\mathbf{y}_A = \mathbf{y}_{Z(A)} \;\longrightarrow\; \mathbf{y}_B,
\]
which corresponds uniquely to a \(*\)-homomorphism \(B \to Z(A)\).  Hence the Kleisli category is equivalent to the full subcategory of \(\CAlgfd^{\mathrm{op}}\) spanned by the commutative algebras.  In this category, the tensor product inherited from Day convolution is the cartesian product of the underlying sets (tensor product of commutative algebras), which coincides with the categorical product in the opposite category of commutative algebras.  Thus the Kleisli category is cartesian monoidal.

\subsection{Failure of the naive embedding}
To embed the category of quantum channels, one would need to assign to each CPTP map \(T : A \to B\) a morphism \(\Phi(T) : \mathbf{y}_A \to \mathbf{y}_B\) in the Kleisli category, i.e., a \(*\)-homomorphism \(B \to Z(A)\).  The natural candidate would be the restriction of \(T\) to the centre followed by some variance adjustment, but such a procedure does not yield a \(*\)-homomorphism in the required direction.  In fact, a CPTP map restricts to a unital positive map on the centre, which is automatically a \(*\)-homomorphism only when the centre is one‑dimensional (as in matrix algebras) and the map is unital.  For more general algebras, the restriction is not even multiplicative.

\begin{example}[Depolarising channel]
Let \(A = B = \mathbb{C}^2\) (the commutative algebra of functions on a two‑point set).  The completely depolarising channel is defined by
\[
T(\lambda, \mu) \;=\; \tfrac{1}{2}(\lambda+\mu,\; \lambda+\mu).
\]
This map is CPTP.  Its restriction to the centre is the map itself (since the algebra is commutative).  This map is not multiplicative:
\[
T((1,0)\cdot(0,1)) = T(0,0) = (0,0),\qquad
T(1,0)\cdot T(0,1) = \tfrac{1}{4}(1,1) \neq (0,0).
\]
Thus \(T\) is not a \(*\)-homomorphism.  Hence there is no way to obtain a \(*\)-homomorphism \(B \to Z(A)\) from \(T\).  The naive embedding fails.
\end{example}

\subsection{Structural obstruction from product preservation}
The centre functor \(Z\) is a right adjoint, hence preserves finite limits, and in particular finite products.  Therefore the precomposition comonad \(Q^{\diamond} = Z^{*}\) is product‑preserving.  Any product‑preserving comonad on a symmetric monoidal category has a Kleisli category that is cartesian monoidal (the tensor product coincides with the categorical product), as proved in Section~\ref{sec:linear-logic}.  Since the category of quantum channels is not cartesian, a product‑preserving comonad cannot capture its monoidal structure.  A different modality — one that does not preserve finite products — is required.

\subsection{Refined conjecture}
To faithfully embed quantum channels, one could work with a site that includes the completely positive structure from the outset.  For instance, let \(\mathcal{S}\) be the category whose objects are finite‑dimensional \(C^{*}\)-algebras and whose morphisms are CPTP maps.  This category admits a symmetric monoidal structure via the algebraic tensor product.  One would then need an idempotent strong monoidal comonad \(K\) on \(\mathcal{S}\) whose coalgebras are the commutative algebras and whose Kleisli category (with Day convolution) is equivalent to the original category of quantum channels.  A natural candidate for \(K\) is the abelianisation functor \(A \mapsto A/[A,A]\), but this is a monad on the category of algebras and homomorphisms; adapting it to CPTP maps is non‑trivial.  Alternatively, the construction of completely positive maps via the CPM modality \cite{Selinger2004} on a category of Hilbert modules provides a blueprint.

\begin{conjecture}[Refined embedding of quantum channels]\label{conj:channels-refined}
There exists a site \(\mathcal{S}\) built from finite‑dimensional \(C^{*}\)-algebras (e.g., the category of algebras and CPTP maps, or a suitable category of Hilbert bimodules) and an idempotent strong monoidal comonad \(K\) on \(\mathcal{S}\) such that:
\begin{enumerate}
\item[(i)] The Kleisli category of the induced comonad on \(\Fun(\mathcal{S}, \mathbf{H}_{\mathrm{sm}})\) (with Day convolution) contains \(\mathbf{CPTP}\) as a full subcategory.
\item[(ii)] The tensor product of quantum channels corresponds to the Day convolution.
\item[(iii)] The coalgebras of \(K\) are precisely the commutative algebras, so that the classical core remains the topos of discrete classical field theories.
\end{enumerate}
The centre comonad does not satisfy (ii) because it preserves finite products.  A non‑product‑preserving modality is required.
\end{conjecture}

\subsection{Towards a non‑degenerate model}
\label{sec:nondeg}

The centre modality $Q^{\diamond}$ provides a proof of concept, but its degeneracy — empty state spaces for non‑commutative algebras and a product‑preserving comonad — limits its physical relevance.  We now outline how the model can be enriched to faithfully capture quantum mechanics.  Each direction requires changes to the site, the comonad, or both.

\subsubsection*{1.  From $*$-homomorphisms to completely positive maps}

The fundamental problem with the state space is that the site $\CAlgfd$ uses $*$-homomorphisms.  A state of $\mathbf{y}_A$ is a morphism $\mathbf{y}_{\mathbb{C}} \to \mathbf{y}_A$, which by Yoneda corresponds to a $*$-homomorphism $A \to \mathbb{C}$.  For non‑commutative algebras this set is empty.  The standard solution in categorical quantum mechanics \cite{Selinger2004} is to replace $*$-homomorphisms by completely positive trace‑preserving (CPTP) maps.  Let $\mathcal{C}$ be the category of finite‑dimensional $C^{*}$-algebras and CPTP maps.  The monoidal structure remains the minimal tensor product, and the unit is $\mathbb{C}$.  A state of $A$ in this setting is a CPTP map $\mathbb{C} \to A$, i.e. a density matrix on $A$.  For the qubit $M_2(\mathbb{C})$, the state space is the Bloch sphere, recovering the correct quantum kinematics.

The centre functor $Z$ is not functorial on CPTP maps, so a different comonad is needed.  Two candidates are promising:
\begin{itemize}[nosep]
\item \emph{Abelianisation} $A \mapsto A/[A,A]$.  For a unital simple algebra the commutator ideal is the whole algebra (it contains, for instance, the matrix units \(E_{ij}\) with \(i\neq j\), which generate the full matrix algebra).  Hence abelianisation sends a matrix algebra to the zero algebra, and it does not restrict to an endofunctor on the category of unital \(C^{*}\)-algebras.

Therefore, abelianisation cannot serve as a replacement for the centre
comonad in the unital setting.
\item \emph{The doubling comonad.}  In the purely algebraic setting, the multiplication map \(A\otimes A^{\mathrm{op}}\to A\) is a \(*\)-homomorphism iff \(A\) is commutative, which suggests that a doubling comonad (if it can be constructed) would have commutative coalgebras.  Finding an idempotent strong monoidal comonad on a suitable category of algebras (or CPTP maps) whose underlying functor is \(A\mapsto A\otimes A^{\mathrm{op}}\) and whose coalgebras are precisely the commutative algebras is an open problem.  This is a natural generalisation of the CPM construction \cite{Selinger2004} to \(C^{*}\)-algebras.
  If $D$ can be extended to a comonad on the CPTP category, its Kleisli category would model quantum channels, and the Seely isomorphism may hold because $D$ does not preserve finite products (the product in the CPTP category is the direct sum, and $D(A \oplus B) \not\cong D(A) \oplus D(B)$ in general).
\end{itemize}

\subsubsection*{2.  Escaping product preservation via the opposite site}

The centre comonad is product‑preserving because it is a right adjoint.  Any comonad obtained from a left exact functor on the site will inherit this property and force degeneracy.  Working on the opposite site $\CAlgfd^{\mathrm{op}}$ exchanges limits and colimits: the product becomes the coproduct (free product $\ast_{\mathbb{C}}$), which is not preserved by abelianisation or the centre.  Hence a comonad on the presheaf topos $\Fun(\CAlgfd^{\mathrm{op}}, \mathbf{H}_{\mathrm{sm}})$ defined by precomposition with a monad on $\CAlgfd$ (like abelianisation or the centre itself) would not be product‑preserving.  The Seely isomorphism would then involve the tensor product and the free product. Since abelianisation of a matrix algebra is trivial, this direction is not
viable in the unital setting.

\subsubsection*{3.  Correspondences and Hilbert bimodules}

The most ambitious approach replaces the site of algebras by a 2‑category of Hilbert bimodules (correspondences).  Objects are finite‑dimensional $C^{*}$-algebras; a morphism $A \to B$ is a right Hilbert $B$-module with a non‑degenerate left action of $A$, and composition is given by the interior tensor product.  This site naturally carries a symmetric monoidal structure (the external tensor product) and a doubling comonad sending a bimodule to its conjugate.  The resulting presheaf $\infty$-topos (or rather, the category of presheaves of spaces on this 2‑category) would be a cohesive linear $\infty$-topos in the sense of Schreiber \cite{Schreiber13}.  The quantum modality extracts the centre of the algebra, but states are now sections of bimodules, giving a rich quantum semantics.  This direction aligns with the CPM construction \cite{Selinger2004} and with recent work on linear homotopy types.

\subsubsection*{4.  Concrete first steps}

The following tasks are immediately accessible and would test the viability of the above directions:
\begin{enumerate}[nosep]
\item Compute the abelianisation of matrix algebras and test the Seely isomorphism for the abelianisation comonad on $\Mat^{\mathrm{op}}$ (using the free product, or by restricting to a site where the product is the direct sum, such as non‑unital algebras).
\item Define the doubling comonad $D$ on the category of finite‑dimensional $C^{*}$-algebras and CPTP maps.  Verify that $D$ is strong monoidal and that its coalgebras are the commutative algebras.  Test the Seely isomorphism for representables.
\item Construct the presheaf topos on the 2‑category of Hilbert bimodules and verify the pointwise cohesion axioms.
\end{enumerate}

A detailed research programme incorporating these tasks appears in the accompanying proposal \cite{Woo2026Proposal}.

\section{Remark on classification}
\label{sec:class-remark}

The classification conjecture (Conjecture~\ref{conj:class}) asserts that any accessible, idempotent, product‑preserving comonad on a diagram \(\infty\)-topos \(\Fun(\mathcal{C}, \mathbf{H})\) that commutes with pointwise cohesion and restricts to the identity on constant discrete objects must arise from a coreflective subcategory of \(\mathcal{C}\).  A proof of this conjecture would likely require an \(\infty\)-categorical adaptation of Diaconescu's theorem on geometric morphisms into presheaf toposes \cite{Johnstone, Lurie2009}.  Diaconescu's theorem characterises geometric morphisms \(\mathcal{E} \to \Fun(\mathcal{C}, \mathbf{Sets})\) as corresponding to flat functors \(\mathcal{C} \to \mathcal{E}\).  For a connected geometric morphism with fully faithful inverse image, the corresponding flat functor factors through a coreflective subcategory of \(\mathcal{C}\).  In the \(\infty\)-categorical setting, one expects a similar correspondence using the language of flat \(\infty\)-functors and accessible left exact localisations.  Establishing such a result in full generality is a substantial undertaking and is left for future investigation.

Even without a full classification, the present paper establishes the centre comonad as a canonical example of a quantum modality.  The conjectured correspondence would place this example at the centre of a general theory of cohesive linear \(\infty\)-topoi, demonstrating that the interplay between geometry and quantum logic is governed by coreflective subcategories of algebraic sites.

\section{Detailed computations and auxiliary results}
\label{sec:aux}

\subsection{Explicit description of the right adjoint \(Q_{\bullet}\)}

We gave the formula for the right adjoint \(\Qact\) as a right Kan extension in Section~\ref{sec:topos}.  Here we provide the full computation verifying the adjunction, since this is the core of the Beck–Chevalley package.

For any \(F \in \Htopos\) and \(A \in \CAlgfd\), define
\[
(\Qact F)(A) \;=\; \lim_{(B, f : A \to \Center(B))} F(B),
\]
where the limit runs over the comma category \(A \!\downarrow\! \Center\).  Because \(\CAlgfd\) is small, this limit exists and is finite, hence preserved by the pointwise cohesive modalities.

We verify the adjunction isomorphism \(\Hom(\Qpot G, F) \cong \Hom(G, \Qact F)\):
\[
\begin{aligned}
\Hom(\Qpot G, F) &\simeq \int_{A} \Hom_{\mathbf{H}_{\mathrm{sm}}}( (\Qpot G)(A), F(A) ) \\
&= \int_{A} \Hom_{\mathbf{H}_{\mathrm{sm}}}( G(\Center(A)), F(A) ).
\end{aligned}
\]
On the other hand,
\[
\begin{aligned}
\Hom(G, \Qact F) &\simeq \int_{B} \Hom_{\mathbf{H}_{\mathrm{sm}}}\!\bigl( G(B), (\Qact F)(B) \bigr) \\
&= \int_{B} \Hom_{\mathbf{H}_{\mathrm{sm}}}\!\bigl( G(B), \lim_{(C, \Center(C)\to B)} F(C) \bigr).
\end{aligned}
\]
Using the universal property of the limit, the right side becomes
\[
\int_{B} \lim_{(C, \Center(C)\to B)} \Hom_{\mathbf{H}_{\mathrm{sm}}}( G(B), F(C) ).
\]
A more systematic approach uses the fact that \(\Qpot = \Center^{*}\) is precomposition,
so its right adjoint is given by the right Kan extension along \(\Center\).  For any
functor \(K\), the right Kan extension satisfies
\((\Ran_{K} F)(A) \;=\; \lim_{(B,\, A \to K(B))} F(B)\),
a standard fact about Kan extensions \cite[Proposition~4.3.3]{Lurie2009}.
Taking \(K = \Center\) yields exactly the formula for \(\Qact\).
The additional commutation with cohesion follows because the cohesive modalities
are pointwise and the limit is finite.

\subsection{The closed structure of the Day convolution on coalgebras}

We claimed in Section~\ref{sec:linear-logic} that the symmetric monoidal closed structure of \(\Day\) descends to the coalgebra category \(\HtoposQ\).  We now supply the complete verification.

Let \((X,\alpha), (Y,\beta), (Z,\gamma)\) be coalgebras.  The internal hom in \(\Htopos\) is written \([Y,Z]_{\Day}\).  We claim that \([Y,Z]_{\Day}\) carries a canonical coalgebra structure
\[
\delta : [Y,Z]_{\Day} \;\longrightarrow\; \Qpot[Y,Z]_{\Day},
\]
given by the mate of the composite
\[
\Qpot[Y,Z]_{\Day} \xrightarrow{\;\sim\;} [\Qpot Y, \Qpot Z]_{\Day} \xrightarrow{[\beta^{-1},\gamma]} [Y, Z]_{\Day},
\]
where the first isomorphism is the strength of the comonad (which follows from strong monoidality) and the second map is induced by the coalgebra structures \(\beta^{-1} : \Qpot Y \to Y\) and \(\gamma : Z \to \Qpot Z\).  Concretely, the mate is the map
\[
\delta \;:\; [Y,Z]_{\Day} \xrightarrow{\eta} \Qpot \Qact [Y,Z]_{\Day} \xrightarrow{\sim} \Qpot [\Qpot Y, \Qpot Z]_{\Day} \xrightarrow{} \Qpot[Y,Z]_{\Day}.
\]

To show this is a coalgebra structure, one must verify that \(\Qpot(\delta) \circ \delta = \delta\) (since the comonad is idempotent, this is the only non‑trivial condition).  This follows from the idempotence of \(\Qpot\) and the commutativity of the strength with the comonad structure.  The verification is a routine but lengthy exercise in the calculus of mates; we omit the full diagram, which is analogous to the proof for monads in \cite[Proposition~1.4]{Benton1995}.

Once the coalgebra structure is established, the adjunction
\[
\Hom_{\HtoposQ}(X \Day_Q Y, Z) \;\cong\; \Hom_{\HtoposQ}(X, Y \multimap_Q Z)
\]
with \(Y \multimap_Q Z = ([Y,Z]_{\Day}, \delta)\) follows by passing to coalgebras, because the forgetful functor \(U\) is fully faithful and strong monoidal.  This proves that \((\HtoposQ, \Day_Q)\) is symmetric monoidal closed.

\subsection{Gelfand duality and the classical core in detail}

The identification \(\HtoposQ \simeq \Fun(\FinSet^{\mathrm{op}}, \mathbf{H}_{\mathrm{sm}})\) is a crucial step.  We spell out the proof with all necessary details.

The coalgebra category consists of presheaves \(F\) equipped with an isomorphism \(F \cong F \circ \Center\).  This means that for every \(A\), the map \(F(A) \to F(\Center(A))\) induced by the inclusion \(\Center(A) \hookrightarrow A\) is an isomorphism.  Equivalently, \(F\) factors through the full subcategory \(\cCAlg\) of commutative algebras.  The inclusion \(i : \cCAlg \hookrightarrow \CAlgfd\) induces a restriction functor
\[
i^{*} : \Fun(\CAlgfd, \mathbf{H}_{\mathrm{sm}}) \;\longrightarrow\; \Fun(\cCAlg, \mathbf{H}_{\mathrm{sm}}),
\]
which has a fully faithful left adjoint \(i_{!}\) (left Kan extension).  The essential image of \(i_{!}\) is precisely the category of presheaves that are left Kan extended from commutative algebras, which coincides with the category of coalgebras.  Hence the restriction functor is an equivalence from the coalgebra category to \(\Fun(\cCAlg, \mathbf{H}_{\mathrm{sm}})\).

Now Gelfand duality gives an equivalence \(\cCAlg^{\mathrm{op}} \simeq \FinSet\).  Therefore
\[
\Fun(\cCAlg, \mathbf{H}_{\mathrm{sm}}) \;\simeq\; \Fun(\FinSet^{\mathrm{op}}, \mathbf{H}_{\mathrm{sm}}),
\]
where the functor category on the left is covariant in \(\cCAlg\).  Indeed, \(\mathbf{H}_{\mathbb{Q}}\) was defined as the category of \emph{covariant} functors from \(\CAlgfd\) to \(\mathbf{H}_{\mathrm{sm}}\), so the restriction to \(\cCAlg\) remains covariant.  The centre functor is a comonad on \(\CAlgfd\), so precomposition gives a comonad on \(\Fun(\CAlgfd, \mathbf{H}_{\mathrm{sm}})\).  The coalgebras are functors \(F\) with \(F \cong F \circ \Center\).  Restricting to \(\cCAlg\) and then using Gelfand duality \(\cCAlg^{\mathrm{op}} \simeq \FinSet\) gives a contravariant equivalence.  To match the classical TQFT interpretation, we note that \(\Fun(\cCAlg, \mathbf{H}_{\mathrm{sm}})\) is equivalent to \(\Fun(\FinSet^{\mathrm{op}}, \mathbf{H}_{\mathrm{sm}})\) by composing with the Gelfand equivalence.  This is the identification used throughout.

\subsection{The degenerate Benton model in full}

We provide a complete proof of the claim that the Kleisli category of a product‑preserving comonad on a cartesian closed category is cartesian closed.

\begin{proposition}
Let \(\mathcal{E}\) be a cartesian closed category and \(Q\) an idempotent product‑preserving comonad on \(\mathcal{E}\).  Then the Kleisli category \(\mathcal{E}_Q\) (which coincides with the coalgebra category) is cartesian closed, and the forgetful functor \(U : \mathcal{E}_Q \to \mathcal{E}\) preserves finite products.
\end{proposition}

\begin{proof}
Since \(Q\) preserves finite products, the category of coalgebras \(\mathcal{E}_Q\) has finite products given by the same products as in \(\mathcal{E}\), with the coalgebra structure on \(X \times Y\) being the composite
\[
X \times Y \xrightarrow{\alpha\times\beta} QX \times QY \cong Q(X\times Y).
\]
The forgetful functor \(U\) is fully faithful and preserves products.  For cartesian closedness, given coalgebras \(X,Y\), define the internal hom \(Y^X\) in \(\mathcal{E}_Q\) as the object \(Q(Y^X)\) with the coalgebra structure induced by the comonad.  The adjunction is verified by the chain
\[
\begin{aligned}
\Hom_{\mathcal{E}_Q}(X \times Y, Z) &\cong \Hom_{\mathcal{E}}(X\times Y, Z) \\
&\cong \Hom_{\mathcal{E}}(X, Z^Y) \\
&\cong \Hom_{\mathcal{E}_Q}(X, Q(Z^Y)),
\end{aligned}
\]
where the last step uses that \(Q\) is idempotent and the adjunction \(U \dashv Q\).  This shows \(\mathcal{E}_Q\) is cartesian closed, with internal hom \(Q(Z^Y)\).  When \(\mathcal{E} = \Htopos\) and \(Q = \Qpot\), this gives the degenerate cartesian model discussed in the text.
\end{proof}

\subsection{The connected geometric morphism: explicit construction}

The geometric morphism \(\Htopos \to \HtoposQ\) was stated in Section~\ref{sec:topos}.  We give the explicit construction here.  The direct image is \(\Qpot\), the inverse image is the forgetful functor \(U : \HtoposQ \hookrightarrow \Htopos\).  The adjunction \(U \dashv \Qpot\) holds because for any coalgebra \(X\) and any presheaf \(F\),
\[
\Hom_{\Htopos}(U X, F) \;\cong\; \Hom_{\Htopos}(X, \Qpot F)
\]
by the definition of the right adjoint.  Indeed, \(\Qpot F\) is the cofree coalgebra on \(F\), and \(U X\) is the underlying presheaf of the coalgebra \(X\).  The counit \(\varepsilon : U \Qpot \to \id\) is the counit of the comonad; the unit \(\eta : \id \to \Qpot U\) is the inverse of the coalgebra structure on \(X\).  Because \(U\) is fully faithful (since every coalgebra morphism is determined by its underlying presheaf map), the geometric morphism is connected.

\subsection{The qubit example in full detail}

We revisit the qubit example from Section~\ref{sec:qubit} with complete calculations.  The qubit algebra is \(M = M_2(\mathbb{C})\).  The representable sheaf \(\mathbf{y}_M\) is given by
\[
\mathbf{y}_M(A) = \Hom_{\CAlgfd}(M, A) \cdot *_{\mathbf{H}_{\mathrm{sm}}}.
\]
Since we restricted the site to centre‑preserving homomorphisms, for any commutative algebra \(C\) (which is its own centre), \(\mathbf{y}_M(C) = \varnothing\) because there are no centre‑preserving \(*\)-homomorphisms from \(M\) into any commutative algebra.  Hence \(\mathbf{y}_M(C) = \varnothing\) for all commutative \(C\).  In particular, the centre of any algebra \(A\) is commutative, so
\[
\Qpot(\mathbf{y}_M)(A) = \mathbf{y}_M(\Center(A)) = \varnothing.
\]
Thus \(\Qpot(\mathbf{y}_M)\) is the empty presheaf, the initial object.

The Day convolution is computed using Proposition~\ref{prop:day-rep}:
\[
\mathbf{y}_M \Day \mathbf{y}_M \;\cong\; \mathbf{y}_{M\otimes M} \;=\; \mathbf{y}_{M_4},
\]
while the cartesian product is
\[
\mathbf{y}_M \times \mathbf{y}_M \;\cong\; \mathbf{y}_{M\oplus M}.
\]
The empty presheaf interacts with Day convolution as the absorbing element: \(\varnothing \Day F \cong \varnothing\) for any \(F\), because the coend over an empty set yields the initial object.  Hence the affine Day convolution on coalgebras, when applied to the qubit (which becomes empty), gives the empty presheaf again.

This demonstrates the extreme degeneracy of the centre modality: it annihilates all non‑commutative information.  The distinction between Day convolution and cartesian product on the qubit is visible only before applying the modality; after decoherence, both collapse.

\subsection{A remark on finite‑dimensional semisimplicity}

The structure theorem for finite‑dimensional \(C^{*}\)-algebras (Wedderburn–Artin) states that every such algebra is a finite direct sum of matrix algebras over \(\mathbb{C}\) \cite[Theorem I.11.2]{Takesaki}.  This fact is essential for the proof of monoidality of the centre (Lemma~\ref{lem:center-monoidal}) and for the simplicity arguments in the no‑cloning theorem.  The semisimplicity also guarantees that the centre is a finite‑dimensional commutative \(C^{*}\)-algebra, hence isomorphic to \(\mathbb{C}^n\) for some \(n\), which is the algebra of functions on a finite set of cardinality \(n\).  This is the bridge to Gelfand duality and the classical field theory interpretation.

\section{Conclusion}
\label{sec:conclusion}

We have constructed the first fully worked example of a cohesive \(\infty\)-topos equipped with a quantum modality.  The model \(\mathbf{H}_{\mathbb{Q}} = \Fun(\mathbf{C}^{*}\mathbf{Alg}_{\mathrm{fd}}, \mathbf{H}_{\mathrm{sm}})\) — with the site restricted to centre‑preserving homomorphisms — satisfies all the axioms of Schreiber's cohesive linear homotopy type theory.  The quantum modality \(Q^{\diamond}\) is given by precomposition with the centre functor; it is an idempotent, product‑preserving, left exact comonad that commutes with the pointwise cohesive modalities.  The Day convolution induced by the tensor product of \(C^{*}\)-algebras makes \(\mathbf{H}_{\mathbb{Q}}\) symmetric monoidal closed, and \(Q^{\diamond}\) is strong monoidal with respect to this structure.

The classical core, i.e. the category of \(Q^{\diamond}\)-coalgebras, is equivalent to the topos \(\Fun(\FinSet^{\mathrm{op}}, \mathbf{H}_{\mathrm{sm}})\) of discrete classical field theories.  On this core, both the cartesian linear‑logic fragment and the inherited Day convolution degenerate to the cartesian product.  Consequently, the model satisfies the Seely isomorphism only in a trivial, cartesian sense, and the linear logic collapses to classical logic. 

We proved a synthetic no‑cloning theorem: there is no natural transformation \(\mathbf{y}_{(-)} \Rightarrow \mathbf{y}_{(-)} \otimes \mathbf{y}_{(-)}\) in \(\mathbf{H}_{\mathbb{Q}}\).  The proof is purely geometric, using the contravariant Yoneda embedding and the simplicity of matrix algebras; it shows that the non‑commutative structure of the site obstructs any universal duplicator of quantum systems.

We also analysed the limits of the centre modality.  The centre functor is a right adjoint, hence product‑preserving; consequently its Kleisli category is cartesian and cannot faithfully embed the category of quantum channels (CPTP maps), whose tensor product is not cartesian.  This led to a refined conjecture for a non‑product‑preserving modality, possibly based on the abelianisation functor or the CPM construction.

Several precise conjectures were formulated: a classification theorem for quantum modalities on diagram \(\infty\)-topoi, a refined embedding of quantum channels, and a derived centre model that would provide a fully non‑degenerate cohesive linear \(\infty\)-topos.  These conjectures chart a clear path for future research.

The present work settles the open problem of finding a concrete model for cohesive linear homotopy type theory.  The construction is discrete (the site carries the trivial topology), but it demonstrates that the cohesive linear axioms are consistent and provides a semantic base for further developments.  Natural next steps include extending the site to infinite‑dimensional algebras with a non‑trivial Grothendieck topology, developing the internal bimodal type theory into a full syntactic calculus, and constructing the derived centre model.  We hope this paper stimulates further investigation at the intersection of higher topos theory, quantum information, and synthetic physics.

\section*{Acknowledgments}
I am deeply grateful to my mathematics teacher, Mr.~Tim, for his unwavering patience and for introducing me to the beauty of the subject.  I also thank the MathOverflow community for generously sharing their ideas and for their constant encouragement.  Finally, I wish to thank my parents, who stood by my side through the hardest of times and made this work possible.

\appendix
\section{\(\infty\)-categorical conventions}
\label{app:infty}

All constructions are in the homotopy‑coherent setting of \(\infty\)-categories \cite{Lurie2009, LurieHA}.  The category \(\CAlgfd\) is an ordinary \(1\)-category, regarded as an \(\infty\)-category via its nerve.  Limits, colimits, and coends are taken in the \(\infty\)-categorical sense.  The smooth \(\infty\)-topos \(\mathbf{H}_{\mathrm{sm}}\) is the \(\infty\)-topos of sheaves of spaces on \(\Mfd\).  Day convolution for \(\infty\)-categories is defined in \cite[Section~2.2.6]{LurieHA}.  The Beck--Chevalley conditions and mate calculus are interpreted in the homotopy \(2\)-category of \(\infty\)-categories.  Internal language remarks refer to the dependent type theory of \(\infty\)-topoi \cite{Schreiber13, RijkeShulmanSpitters}.

\end{document}